\documentclass[1p]{elsarticle}
\usepackage{amssymb}

\newtheorem{theorem}{Theorem}
\newdefinition{remark}{Remark}
\newproof{proof}{Proof}

\begin{document}
\begin{frontmatter}
\title{Controllability properties for some semilinear parabolic PDE with
a quadratic gradient term \tnoteref{t1}}
\tnotetext[t1]{Partially supported by
Spanish Ministry of Science and Innovation under grant MTM2011-22711.}
\author{Luis A. Fern\'andez}
\ead{lafernandez@unican.es}
\address{Dep. Matem\'aticas, Estad\'{\i}stica y Computaci\'on, Universidad de
Cantabria. \\ Avda. de los Castros, s/n, 39071, Santander (Spain).}

\begin{abstract}  We study several controllability properties
for some semilinear parabolic PDE with a quadratic gradient term.
For internal distributed controls, it is shown that the system is
approximately and null controllable. The
proof relies on the Cole-Hopf transformation. The same approach is used to deal with
initial controls.
\end{abstract}



\end{frontmatter}

\section{Introduction}

We study some controllability properties for the following
semilinear parabolic boundary value problem

\begin{equation}
\left\{
\begin{array}{rcll}
y_t(x,t) - \Delta y(x,t) & =  & |\nabla y(x,t)|^2 \cdot
 \phi(y(x,t)) + u(x,t)\chi_{\omega}(x) & \mbox{in } Q_T,
\\ y(x,t) & = & 0   & \mbox{on }
\Sigma_T,  \\ y(x,0) & =  & y_0(x)  & \mbox{in } \Omega,
\end{array}
\right. \label{E1}
\end{equation}
where $Q_T = \Omega \times (0,T)$ and $ \Sigma_T = \Gamma \times
(0,T)$, $\Omega$ being an open bounded domain of $\mathbb R^n$,
with boundary $\Gamma$ of class $C^2$ and $T$ a positive (fixed)
real number. Moreover, $\omega$ is a nonempty open subset of
$\Omega$ (possibly small), $\chi_{\omega}$ is the characteristic function of $\omega$ and
$\phi$ is a real function. As usual, we will denote $y_t(x,t) =
\frac{\partial y}{\partial t}(x,t),$ $\nabla y(x,t) = \left(
\frac{\partial y}{\partial x_1}(x,t),\ldots,\frac{\partial
y}{\partial x_n}(x,t)\right)$ and $\Delta y = \sum_{i =
1}^n \frac{\partial^2 y}{\partial x_i^2}(x,t)$.

Given $T > 0$ fixed, it is said that problem (\ref{E1}) is
approximately controllable in $L^2(\Omega)$ at time $T$ by using
controls in ${\cal U}$ if, for each initial condition $y_0$ in
certain space, $y_d \in L^2(\Omega)$ and $\epsilon
> 0$, there exists $u \in {\cal U}$ and a solution $y$ of the 
problem (\ref{E1}) that satisfies $ \|y(\cdot,T) - y_d\|_{L^2(\Omega)} <
\epsilon$. Furthermore, it is said that problem (\ref{E1}) is null
controllable at time $T$ by using controls in ${\cal U}$ when for
each initial condition $y_0$, there exists $u \in {\cal U}$ and 
a solution $y$ of the problem (\ref{E1}) satisfying $y(x,T) = 0$ in
$\Omega$. 

In the last years, there has been a great interest in the
controllability of parabolic systems with internal distributed
control (among others, see \cite{Barbu00}, \cite{DFCGBZ02}--\cite{Khapa01} and the
references therein). Mainly, these works try to characterize the
class of parabolic problems having the mentioned controllability
properties. By now, it is well known that the approximate controllability
holds for most of the linear parabolic PDE. In the nonlinear case, 
many results are related with the system

\begin{equation}
\left\{
\begin{array}{rcll}
y_t(x,t) - \Delta y(x,t) + f(y(x,t),\nabla y(x,t)) & =  & u(x,t)\chi_{\omega}(x) & \mbox{in }
Q_T,
\\ y(x,t) & = & 0   & \mbox{on }
\Sigma_T,  \\ y(x,0) & =  & y_0(x)  & \mbox{in } \Omega.
\end{array}
\right. \label{E2}
\end{equation}

Roughly speaking, the situation can be described as follows:

\begin{itemize}
\item for $n=1$ and $f(y,y_x) = y y_x$ (i.e. Burgers equation), it
is known that the system (\ref{E2}) is not approximately
controllable in $L^2(\Omega)$ at time $T$ by using controls in
$L^2(\omega \times (0,T))$, see \cite[Theorem 6.3, Chapter
I]{Fursi-Imanu96}. \item for functions $f(y,\nabla y)$ globally
Lipschitz with respect to $(y,\nabla y)$, the
system (\ref{E2}) is approximately controllable in
$H_0^\rho(\Omega)$ at time $T$ for all $\rho \in [0,1)$, see \cite[Theorem 3.3]{Fdez-Zuazua99},
and null controllable at time $T$ (when $f(0,0)=0$), see \cite{Imanu-Yamamo03}, by using controls in
$L^2(\omega \times (0,T))$.  \item
for functions $f(y,\nabla y)$ growing slower than
\begin{equation}
|y|\log^{3/2} {(1+|y|+|\nabla y|)} + |\nabla y|\log^{1/2}
{(1+|y|+|\nabla y|)},
\label{E3}
\end{equation} as $|(y,\nabla y)|
\rightarrow +\infty$, the system (\ref{E2}) is
null (when $f(0,0)=0$) and approximately controllable in $L^2(\Omega)$ at time $T$,
by using controls in $L^\infty(\omega \times (0,T))$, see
\cite{DFCGBZ02}.
\end{itemize}

In general, the proofs of these results are quite technical and
involve Carleman estimates together with sharp parabolic
regularity results.

The main contribution of this paper is presented in Theorem \ref{T1}, where it is shown that (\ref{E1})
is approximately controllable in $L^2(\Omega)$ and null
controllable at time $T$ by using controls in $L^\infty(\omega \times
(0,T))$. 

\section{Internal distributed control}

Let us begin by recalling some spaces that appear commonly in the
framework of parabolic problems:
\[ W^{1,1}_{2,0}(Q_T) = \{y \in L^2(0,T;H_0^1(\Omega)) : y_t \in
L^2(Q_T)\}, \]
\[ W^{2,1}_{2,0}(Q_T) = \{y \in L^2(0,T;H^2(\Omega)
\cap H_0^1(\Omega)) : y_t \in L^2(Q_T) \}. \]

When $\Gamma$ is $C^2$, it is known that $W^{2,1}_{2,0}(Q_T) = \{y
\in L^2(0,T;H_0^1(\Omega)) : y_t \in L^2(Q_T), \ \Delta y \in
L^2(Q_T) \}$ (see \cite[p. 109 and 113]{Lad85}). It is also well known that
\[ W^{1,1}_{2,0}(Q_T) \subset C([0,T];L^2(\Omega)), \ \  \ W^{2,1}_{2,0}(Q_T) \subset C([0,T];H_0^1(\Omega)), \]
with continuous imbeddings.

A crucial assumption along this paper is the following:
\begin{itemize}
\item [{\bf (H)}] $\phi : \mathbb R \longrightarrow \mathbb R$ is
a continuous function and there exists a real number $\alpha \leq
0$ such that
\[ \int_{0}^r \phi(s) ds \geq \alpha, \ \ \forall r \in \mathbb R.\]
\end{itemize}

Now, let us introduce the function $\varphi: \mathbb R \longrightarrow \mathbb R$ given by
\begin{equation}
\varphi(r) = \int_0^r \exp{\left(\int_{0}^v \phi(s) ds \right)}
dv, \ \ \forall r \in \mathbb R. \label{E10}
\end{equation}

Under condition {\bf (H)}, it is straightforward to show that $\varphi$ is an strictly
increasing $C^2$ function with range equals to $\mathbb R$, thanks to 
\begin{equation}
\varphi'(r) = \exp{\left(\int_{0}^r \phi(s) ds \right)} \geq
\exp{(\alpha)}> 0, \ \  \mbox{ and } \ \ \varphi''(r) = \phi(r)\varphi'(r), \ \ \forall r \in \mathbb R.
\label{E20}
\end{equation}

We will consider the problem (\ref{E1}) with control $u \in L^\infty(Q_T)$ and initial condition $y_0 \in H_0^1(\Omega) \cap L^\infty(\Omega)$. A function $y \in W^{1,1}_{2,0}(Q_T) \cap L^\infty(Q_T)$ is said to be a solution of (\ref{E1}) if it verifies its PDE in the distribution sense and the initial condition in $L^2(\Omega)$. 

Let us show that the problem (\ref{E1}) can be transformed into a semilinear one by using the Cole-Hopf transformation (\ref{E10}):
given a solution $y$ of problem (\ref{E1}), we define a new function given by 
\begin{equation}
z(x,t) = \varphi(y(x,t)). \label{E35}
\end{equation}
It is easy to show that $z \in W^{2,1}_{2,0}(Q_T) \cap
L^\infty(Q_T)$, because $\varphi(0) = 0$ and
\begin{equation}
z_t(x,t) = \varphi'(y(x,t)) \cdot y_t(x,t), \ \ \nabla z(x,t) =
\varphi'(y(x,t)) \cdot \nabla y(x,t). \label{E40}
\end{equation}
Furthermore, taking into account (\ref{E20}), it
follows that (in the distribution sense) 
\[ \Delta z(x,t) = \varphi''(y(x,t)) |\nabla y(x,t)|^2+
\varphi'(y(x,t)) \Delta y(x,t) \]
\begin{equation}
=\varphi'(y(x,t)) \left(\phi(y(x,t)) |\nabla y(x,t)|^2 + \Delta
y(x,t) \right) \label{E50}
\end{equation}
\[ =\varphi'(y(x,t))\left(y_t(x,t) - u(x,t)\chi_{\omega}(x)\right)
= z_t(x,t) -
u(x,t)\chi_{\omega}(x)\varphi'(y(x,t)).\]

Therefore, $z$ can be viewed as a solution in $W^{2,1}_{2,0}(Q_T)
\cap L^\infty(Q_T)$ of the semilinear problem

\begin{equation}
\left\{
\begin{array}{rcll}
z_t(x,t) - \Delta z(x,t) & =  &  u(x,t)\chi_{\omega}(x) \cdot
\varphi'(\varphi^{-1}(z(x,t)))  & \mbox{in } Q_T,
\\ z(x,t) & = & 0   & \mbox{on }
\Sigma_T,  \\ z(x,0) & =  & \varphi(y_0(x))  & \mbox{in } \Omega.
\end{array}
\right. \label{E60}
\end{equation}

\begin{remark}
The existence of solution for problem (\ref{E1}) has been studied in many papers, even for less 
regular data $u$ and $y_0$ (see \cite[Theorem 5.6]{ADP08}), and 
for more general operators (see \cite[Theorem 2.2]{DGS07}). Let us stress that we are dealing with more regular solutions than those obtained in these works. On the other hand, it was proved in  \cite{ADP08} that the problem (\ref{E1}) admits infinitely many weak solutions. Clearly, this nonuniqueness result is not relevant from the controllability view point, where only one solution satisfying the required conditions is required.
\end{remark}

Now, we can take advantage of the Cole-Hopf transformation in order to obtain
the main result of this section

\begin{theorem}
Let us assume condition {\bf (H)} and
$y_0 \in H_0^1(\Omega) \cap L^\infty(\Omega)$. Then, the system
(\ref{E1}) is approximately controllable in $L^2(\Omega)$ and null controllable at time
$T$ by using controls in $L^\infty(\omega \times (0,T))$. \label{T1}
\end{theorem}

\begin{proof}
Let us begin by proving the approximate controllability property
for (\ref{E1}). Due to the density of $L^\infty(\Omega)$ in
$L^2(\Omega)$, it is enough to approximate any element $y_d \in
L^\infty(\Omega)$. Given $\epsilon
> 0$, by the approximate controllability
for the Heat equation (see \cite[section 10, chapter III]{Lions71} and \cite{Lions90}), we can guarantee the
existence of a control $v \in L^\infty(Q_T)$ such
that the unique solution $z$ in $W^{2,1}_{2,0}(Q_T) \cap
L^\infty(Q_T)$ (see \cite[p. 112]{Lad85} and
\cite[Theorem 7.1, p. 181 and Corollary 7.1, p. 186]{Lad-Sol-Ura68}) of the linear problem

\begin{equation}
\left\{
\begin{array}{rcll}
z_t(x,t) - \Delta z(x,t) & =  &  v(x,t)\chi_{\omega}(x) &
\mbox{in } Q_T,
\\ z(x,t) & = & 0   & \mbox{on }
\Sigma_T,  \\ z(x,0) & =  & \varphi(y_0(x))  & \mbox{in } \Omega,
\end{array}
\right. \label{E70}
\end{equation}
satisfies $\|z(\cdot,T) - \varphi(y_d)\|_{L^2(\Omega)} \leq \epsilon
\cdot \exp{(\alpha)},$ where $\alpha$ is taken from {\bf (H)}.
Inspired by the argumentation developed at the beginning of this
section (that we are reversing now), we define
\begin{equation}
y(x,t) = \varphi^{-1}(z(x,t)) \ \ \mbox{ in } \
Q_T.\label{E75}
\end{equation} Of
course, $y$ is well defined: $\varphi^{-1}(s)$ exists for all $s
\in \mathbb R$, because $\varphi$ is an strictly increasing
function with range equals to $\mathbb R$. In fact, thanks to
condition {\bf (H)}, we know that $\varphi^{-1}(s)$ is a globally
Lipschitz increasing $C^2$ function with
\begin{equation}
\frac{d \varphi^{-1}}{ds}(s) = \frac{1}{\varphi'(\varphi^{-1}(s))}
\in (0,\exp{(-\alpha)}], \ \ \frac{d^2 \varphi^{-1}}{ds^2}(s) =
\frac{-\varphi''(\varphi^{-1}(s))}{(\varphi'(\varphi^{-1}(s)))^3}.
\label{E80}
\end{equation}

Furthermore, it is easy to show that $y \in W^{1,1}_{2,0}(Q_T)
\cap L^\infty(Q_T),$ combining that $z
\in W^{2,1}_{2,0}(Q_T) \cap L^\infty(Q_T)$, $\varphi(0) = 0$ and
\begin{equation}
y_t(x,t) = \frac{d \varphi^{-1}}{ds}(z(x,t)) \cdot z_t(x,t), \ \
\nabla y(x,t) = \frac{d \varphi^{-1}}{ds}(z(x,t)) \cdot \nabla
z(x,t). \label{E90}
\end{equation}
Taking into account (\ref{E20}) and (\ref{E70})--(\ref{E90}), the next equalities 
follow in the distribution sense:
\[\Delta y(x,t) = \frac{d^2 \varphi^{-1}}{ds^2}(z(x,t)) |\nabla
z(x,t)|^2+ \frac{d \varphi^{-1}}{ds}(z(x,t)) \Delta z(x,t) \]
\[= \frac{-\varphi''(y(x,t))}{(\varphi'(y(x,t)))^3}|\nabla z(x,t)|^2
+ \frac{\Delta z(x,t)}{\varphi'(y(x,t))} = -\phi(y(x,t))\frac{|\nabla z(x,t)|^2}{(\varphi'(y(x,t)))^2} \]
\[ + \frac{z_t(x,t) - v(x,t)\chi_\omega(x)}{\varphi'(y(x,t))} = -\phi(y(x,t))|\nabla y(x,t)|^2 + y_t(x,t) - u(x,t)\chi_\omega(x), \]
where we have selected the control

\begin{equation}
u(x,t) = \frac{v(x,t)}{\varphi'(y(x,t))}.
\label{E105}
\end{equation}

Obviously, $u$ belongs to $L^\infty(Q_T)$, thanks
to (\ref{E20}):
\[ \|u\|_{L^{\infty}(Q_T)} \leq \|v\|_{L^{\infty}(Q_T)} \cdot \exp{(-\alpha)}. \]

This means that $y$ is a solution of problem (\ref{E1}). Finally, combining the
Mean Value Theorem, (\ref{E20}) and (\ref{E80}), we obtain
\[ \|y(\cdot,T) -
y_d\|_{L^2(\Omega)} = \|\varphi^{-1}(z(\cdot,T)) -
\varphi^{-1}(\varphi(y_d))\|_{L^2(\Omega)} \]
\[ = \left(\int_{\Omega} \left|\frac{d \varphi^{-1}}{ds}(\theta(x))\right|^2
|z(x,T)-\varphi(y_d(x))|^2 dx\right)^{1/2}  \leq \exp{(-\alpha)}
\|z(\cdot,T) - \varphi(y_d)\|_{L^2(\Omega)} \leq \epsilon,\] where
$\theta(x)$ denotes some intermediate value between $z(x,T)$ and
$\varphi(y_d(x))$. This is exactly what we were looking for.

The proof of the null controllability property can be seen as a
particular case of the previous argumentation, taking $\epsilon =
0$, $y_d = 0$ and applying the null controllability result for the
Heat equation with bounded controls, see \cite[Theorem 3.1]{DFCGBZ02} and also \cite{Leb-Rob95}.

\end{proof}

Of course, there exist many continuous functions $\phi$ verifying
condition {\bf (H)}. Typical examples are $\phi(y) = \exp{(y)}$ (with $\alpha =-1$) and $\phi(y) = y^{2k+1}$ for any natural number $k$ (with $\alpha =0$). More generally, $\phi(y)$ can be any polynomial with highest
term of odd order and positive main coefficient. It is also clear that some other usual functions do not verify
condition {\bf (H)}, like $\phi(y) = y^{2k}$ for any natural
number $k$. These cases deserve a specific treatment: for instance, the case $\phi(y) = 1$ was studied in \cite{Lei-Li-Lin08}. 

From Theorem \ref{T1} it follows that the hypotheses assumed in \cite{DFCGBZ02} (see (\ref{E3})) are far
from being necessary to derive the controllability properties, because clearly they are not satisfied in our framework. 

\section{Initial control}

Previous argumentation can be also applied
when the control is acting through the initial condition, like in the problem:

\begin{equation} \left\{
\begin{array}{rcll}
y_t(x,t) - \Delta y(x,t) & =  & |\nabla y(x,t)|^2 \cdot
 \phi(y(x,t)) & \mbox{in } Q_T,
\\ y(x,t) & = & 0   & \mbox{on }
\Sigma_T,  \\ y(x,0) & =  & u(x)  & \mbox{in } \Omega.
\end{array}
\right. \label{E5}
\end{equation}

Similarly to the previous case, it is said that problem (\ref{E5}) is
approximately controllable in $L^2(\Omega)$ at time $T$ by using
initial controls $u$ in certain space ${\cal U}$ if, for each $y_d \in L^2(\Omega)$ and $\epsilon
> 0$, there exists $u \in {\cal U}$ and a solution $y$ of the 
problem (\ref{E5}) that satisfies $ \|y(\cdot,T) - y_d\|_{L^2(\Omega)} <
\epsilon$.

The following result can be proved in this context:

\begin{theorem}
Let us assume condition {\bf (H)}. Then, the
system (\ref{E5}) is approximately controllable in $L^2(\Omega)$
at time $T$ by using controls $u \in H_0^1(\Omega) \cap
L^\infty(\Omega)$. \label{T3}
\end{theorem}

\begin{proof}
Given any element $y_d \in L^\infty(\Omega)$ and $\epsilon
> 0$, by the approximate controllability
proved in \cite[Th. 11.2, pg. 215]{Lions71}, we can guarantee the existence of a
control $v \in H_0^1(\Omega) \cap L^\infty(\Omega)$ such that the
unique solution $z$ in $W^{2,1}_{2,0}(Q_T) \cap L^\infty(Q_T)$
of the linear problem 

\begin{equation}
\left\{
\begin{array}{rcll}
z_t(x,t) - \Delta z(x,t) & =  &  0  & \mbox{in } Q_T,
\\ z(x,t) & = & 0   & \mbox{on }
\Sigma_T,  \\ z(x,0) & =  & v(x)  & \mbox{in } \Omega,
\end{array}
\right. \label{E225}
\end{equation}
satisfies $\|z(\cdot,T) - \varphi(y_d)\|_{L^2(\Omega)} \leq \epsilon \cdot \exp{(\alpha)}.$ We finish the proof as in Theorem \ref{T1}, by selecting the control
\begin{equation}
u(x) = \varphi^{-1}(v(x)), \label{E230}
\end{equation}
that belongs to $H_0^1(\Omega)
\cap L^\infty(\Omega)$, thanks to the properties of $\varphi$.
\end{proof}

\section*{Acknowledgment}
The author is grateful to professors I. Peral and S. Segura de Le\'on for helpful comments about the references 
\cite{ADP08} and \cite{DGS07}.

\end{document}